\documentclass[reqno]{amsart}

\usepackage{amsmath,amsfonts,amssymb,amsthm,diagrams}
\usepackage{cancel}

\usepackage{KIRstyle}

\theoremstyle{plain}
\newtheorem{prop}{Proposition}[section]
\newtheorem{lem}[prop]{Lemma}
\newtheorem{conj}[prop]{Conjecture}
\newtheorem{cor}[prop]{Corollary}
\newtheorem{thm}[prop]{Theorem}

\theoremstyle{definition}

\theoremstyle{remark}
\newtheorem{remark}[prop]{Remark}
\newtheorem{example}[prop]{Example}

\title[{G}rothendieck classes of quiver cycles as iterated residues]{{G}rothendieck classes of quiver cycles as iterated residues}
\author{Justin Allman}
\address{Department of Mathematics, UNC--Chapel Hill \\ Phillips Hall CB\#3250 \\ Chapel Hill, NC 27599--3250}
\email{jallman@email.unc.edu}

\usepackage[
	bookmarksnumbered,
	pdftitle={{G}rothendieck classes of quiver cycles as iterated residues}
	pdfauthor={Justin Allman}
	pdfsubject={MSC Classification: (Primary) 14N15, (Secondary) 16E20, 16G20, 05E99}
	]
	{hyperref}

\begin{document}

\maketitle

\begin{abstract}
In the case of Dynkin quivers we establish a formula for the {G}rothendieck class of a quiver cycle as the iterated residue of a certain rational function, for which we provide an explicit combinatorial construction. Moreover, we utilize a new definition of the double stable Grothendieck polynomials due to Rim\'{a}nyi and Szenes in terms of iterated residues to exhibit how the computation of quiver coefficients can be reduced to computing coefficients in Laurent expansions of certain rational functions. 
\end{abstract}

\section*[Introduction]{Introduction}\label{s:Intro}

Let $Q$ be a quiver with a finite vertex set $Q_0 = \{1,\ldots,N\}$ and finite set of arrows $Q_1$, each of which has a \emph{head} and \emph{tail} in $Q_0$. For $a\in Q_1$, these vertices are denoted $h(a)$ and $t(a)$ respectively. Throughout the sequel we will refer also to the set
\begin{equation}\label{eqn:T(i)}
T(i) = \{j\in Q_0 \,|\, \exists a \in Q_1 \text{~with~} h(a)=i \text{~and~} t(a)=j\}.
\end{equation}
Given a \emph{dimension vector} of non-negative integers $\al{v}=(v_1,\ldots,v_N)$, define vector spaces $E_i = \C^{v_i}$ and the affine \emph{representation space} $V = \Dirsum_{a\in Q_1} \Hom(E_{t(a)},E_{h(a)})$ with a natural action of the algebraic group $\al{G}=GL(E_1)\cross\cdots\cross GL(E_N)$ given by
\begin{equation}
\label{eq:Gact}(g_i)_{i\in Q_0}\cdot(\phi_a)_{a\in Q_1} = (g_{h(a)}\phi_a g_{t(a)}^{-1})_{a\in Q_1}.
\end{equation}
A \emph{quiver cycle} $\Omega \subset V$ is a $\al{G}$-stable, closed, irreducible subvariety and, as such, has a well defined structure sheaf $\curly O_\Omega$. The goal of this paper is the calculation of the class $$[\curly O_{\Omega}] \in K_{\al{G}}(V),$$ in the $\al{G}$-equivariant Grothendieck ring of $V$. To accomplish this, we reformulate the problem in an equivalent setting; we realize $[\curly O_\Omega]$ as the $K$-class associated to a certain degeneracy locus of a quiver of vector bundles over a smooth complex projective base variety $X$.

Formulas for this class exist already in the literature, the most general of which is due to Buch \cite{BuchQCDT}, and which we now explain. Buch's result is given in terms of the stable version of Grothendieck polynomials first invented by Lascoux and Sch\"{u}tzenberger as representatives of structure sheaves of Schubert varieties in a flag manifold \cite{LaSch} which are applied to the $E_i$ in an appropriate way. For a comprehensive introduction to the role of Grothendieck polynomials in $K$-theory, see \cite{BuchCKT}.

The stable Grothendieck polynomials $G_\lambda$ are indexed by partitions, i.e.\ non-increasing sequences of non-negative integers $\lambda = (\lambda_1\geq \lambda_2 \geq \cdots )$ with only finitely many parts nonzero. The number of nonzero parts is called the \emph{length} of the partition and denoted $\ell(\lambda)$. For each $i\in Q_0$, form the vector space $M_i = \Dirsum_{j\in T(i)} E_j$. With this notation, Buch shows that for \emph{unique} integers $c_\mu(\Omega)\in \Z$ one has
\begin{equation}\label{eqn:quivcoef}
[\curly O_\Omega] = \sum_{\mu} c_\mu(\Omega) G_{\mu_1}( E_1- M_1)\cdots G_{\mu_N}( E_N -  M_N) \in K_{\al{G}}(V)
\end{equation}
where the sum is taken over all sequences of partitions $\mu=(\mu_1,\ldots,\mu_N)$ subject to the constraint that $\ell(\mu_i) \leq v_i$ for all $1\leq i \leq N$. The integers $c_\mu(\Omega)$ are called the \emph{quiver coefficients}. In the case that $Q$ is a \emph{Dynkin quiver}, that is, its underlying non-oriented graph is one of the simply-laced Dynkin diagrams (of type $A$, $D$, or $E$), Buch shows that the sum above is finite. The central question in the theory is, if one assumes that $\Omega$ has rational singularities, are the quiver coefficients \emph{alternating}? In this setting, alternating is interpreted to mean that $(-1)^{|\mu|-\codim(\Omega)}c_\mu(\Omega) \geq 0$ for all $\mu$, where $|\mu|=\sum_i |\mu_i|$ and $|\mu_i|$ is the area of the corresponding Young diagram. An answer to this question supersedes many of the other positivity conjectures in this vein, in particular, whether or not the cohomology class $[\Omega]\in H^{*}_{\al{G}}(V)$ is Schur positive, since the leading term of $G_\lambda$ is the Schur function $s_\lambda$ and the coholomology class $[\Omega]$ can be interpreted as a certain leading term of the $K$-class $[\curly O_\Omega]$. For this reason, the quiver coefficients $c_\mu(\Omega)$ for which $|\mu| = \codim(\Omega)$ are called the \emph{cohomological quiver coefficients}.

The goal of this paper is to give a new formula for $[\curly O_\Omega]$ in terms of iterated residue operations. The motivation is plain--namely there has been some considerable recent success in attacking positivity and stability results in analogous settings once armed with such a formula.

In \cite{FRLMS}, Feh\'{e}r and Rim\'{a}nyi discover that Thom polynomials of singularities share unexpected stability properties, and this is made evident through non-conventional generating sequences. The ideas of \cite{FRLMS} are further developed and organized in \cite{BS}, \cite{FRAnn}, and \cite{KazTP} where the generating sequence formulas appear under the name \emph{iterated residue}. In particular, in \cite{BS} B\'{e}rczi and Szenes prove new positivity results for certain Thom polynomials, and Kazarian is able to calculate new classes of Thom polynomials in \cite{KazTP} through iterated residue machinery developed in \cite{KazGHD}.

Even more recently, a new formula for the cohomology class of the quiver cycle $[\Omega]\in H^{*}_{\al{G}}(V)$ as an iterated residue has been reported in \cite{qr3}, and some new promising initial results on Schur positivity have been obtained from this formula in \cite{Kal}. Moreover in \cite{coha}, Rim\'{a}nyi describes an explicit connection between the iterated residue formula for cohomological quiver coefficients of \cite{qr3} and certain structure constants in the Cohomological Hall algebra (COHA) of Kontsevich and Soilbelman \cite{KScoha}.

The organization of the paper is as follows. In Section \ref{s:QRDL} we describe quiver representations in some more detail and define the degeneracy loci associated to them. In Section \ref{s:RS} we discuss an algorithm of Reineke to resolve the singularities of the degeneracy loci in question, which produces a sequence of well-understood maps that we eventually utilize for our calculations. In Section \ref{s:IR} we define our iterated residue operations and provide some illustrative examples of their application. In Section \ref{s:MT} we present the statement of the main result and by example, compare our method to previous formulae, most notably that of \cite{BuchQCDT} and the cohomological iterated residue formula from \cite{qr3}. In Section \ref{s:ELIR} we describe how the push-forward (or Gysin) maps associated to Grassmannian fibrations are calculated with equivariant localization and translated to the language of iterated residues, and in Section \ref{s:PMT} we provide the proof of the main theorem. In Section \ref{s:GPIR} we use a new definition of Grothendieck polynomials proposed by Rim\'{a}nyi and Szenes to exhibit that our formula produces an explicit rational function whose coefficients, once expanded as a multivariate Laurent series, correspond to the quiver coefficients. We expect that further analysis of these rational functions will produce new positivity results regarding the quiver coefficients.

The author thanks Alex Fink, Ryan Kaliszewski, and Rich\'{a}rd Rim\'{a}nyi for helpful conversations related to this topic and Merrick Brown for computational advice.

\section[Quiver reps.\ and deg.\ loci]{Quiver representations and degeneracy loci}\label{s:QRDL}

\subsection{Quiver cycles for Dynkin quivers}\label{ss:QR}
In this paper we will consider only Dynkin quivers, which always have finite sets of vertices and arrows, and contain no cycles. Throughout the sequel, $Q$ denotes a Dynkin quiver with vertices $Q_0=\{1,\ldots,N\}$ and arrows $Q_1$, $\al{v}=(v_1,\ldots,v_N)\in \N^N$ denotes a dimension vector and $V$ denotes the corresponding representation space.

Let $\Omega$ be a quiver cycle. For technical reasons, we henceforth assume that $\Omega$ is Cohen-Macaulay with rational singularities. In the case of Dynkin quivers, Gabriel's theorem \cite{Gab} implies that there are only finitely many stable $\al{G}$-orbits and as a consequence, every quiver cycle must be a $\al{G}$-orbit closure (and conversely). Moreover, the orbits have an explicit description, as follows.

Let $\{\varphi_i : 1\leq i \leq N\}$ denote the set of simple roots of the corresponding root system and $\Phi^+$ the set of positive roots. For any positive root $\varphi$, one obtains integers $d_1(\varphi),\ldots,d_N(\varphi)$ defined uniquely by $\varphi = \sum_{i=1}^N d_i(\varphi)\varphi_i$. The $\al{G}$-orbits in $V$ are in one-to-one correspondence with vectors \begin{equation} m=(m_\varphi)\in \N^{\Phi^+}, \;\;\; \text{such that} \;\;\; \sum_{\varphi\in\Phi^+}m_\varphi d_i(\varphi) = v_i, \; \; \text{for each $1 \leq i \leq N$}.\nonumber \end{equation} Observe that the list of orbits does not depend on the orientation of the arrows of $Q$ but only on the underlying non-oriented graph. Throughout the sequel, we will denote the orbit-closure corresponding to $m \in \N^{\Phi^+}$ by $\Omega_m$.

\subsection{Degeneracy loci associated to quivers}\label{ss:DL}

Let $X$ be a smooth complex projective variety, and let $K(X)$ denote the Grothendieck ring of algebraic vector bundles over $X$. A \emph{$Q$-bundle} $(\curly E_\bullet,f_\bullet) \to X$ is the following data: \begin{itemize} \item for each $i\in Q_0$ a vector bundle $\curly E_i\to X$ with $\rk(\curly E_i) = v_i$, and \item for each arrow $a\in Q_1$, a map of vector bundles $f_a:\curly E_{t(a)} \to \curly E_{h(a)}$ over $X$. \end{itemize} Let $(\curly E_\bullet,f_\bullet)_x$ denote the \emph{fiber} of the $Q$-bundle at the point $x\in X$; this consists of vector spaces $(\curly E_1)_x,\ldots,(\curly E_N)_x$ (the fibers of the vector bundles) and also a linear map $(f_a)_x:(\curly E_{t(a)})_x\to(\curly E_{h(a)})_x$ for each $a\in Q_1$. Corresponding to the quiver cycle $\Omega \subset V$, define the \emph{degeneracy locus}
\begin{equation}\label{eqn:dldefn}
\Omega(\curly E_\bullet) = \{x \in X \; | \; (\curly E_\bullet,f_\bullet)_x \in \Omega\}.
\end{equation}
Observe that the fiber $(\curly E_\bullet,f_\bullet)_x$ only belongs to $V=\Dirsum_{a\in Q_1}\Hom(E_{t(a)},E_{h(a)})$ once one specifies a basis in each vector space $(\curly E_i)_x$. However, the degeneracy locus above is well-defined since the action of $\al{G}$ on $V$ described by equation (\ref{eq:Gact}) can interchange any two choices for bases, and $\Omega$ is $\al{G}$-stable. The relevance of the degeneracy locus $\Omega(\curly E_\bullet)$ is

\begin{prop}[Buch]
If $X$ and $\Omega$ are both Cohen-Macaulay and the codimension of $\Omega(\curly E_\bullet)$ in $X$ is equal to the codimension of $\Omega$ in $V$, then $$[\curly O_{\Omega(\curly E_\bullet)}] = \sum_{\mu} c_\mu(\Omega) G_{\mu_1}(\curly E_1- \curly M_1)\cdots G_{\mu_N}(\curly E_N - \curly M_N) \in K(X)$$ where $\curly M_i = \Dirsum_{j\in T(i)}\curly E_j$ and the $c_\mu(\Omega)$ are exactly the quiver coefficients defined by Equation (\ref{eqn:quivcoef}).
\end{prop}

The hypothesis of the above result is the reason for our technical assumption that $\Omega$ be Cohen-Macaulay. The goal of this paper is to give a new formula for the class corresponding to the structure sheaf of $\Omega(\curly E_\bullet)$ in the Grothendieck ring $K(X)$, and hence by the uniqueness of the quiver coefficients, a new formula for $[\curly O_\Omega] \in K_{\al{G}}(V)$.

\begin{remark}[Notation and genericity]\label{rem:genericity}
A choice of maps $f_\bullet$ for a $Q$-bundle amounts to a section of $\curly V = \Dirsum_{a\in Q_1}\Hom(\curly E_{t(a)},\curly E_{h(a)})$ and the choices $f_\bullet$ for which the degeneracy locus $\Omega(\curly E_\bullet)$ has its expected codimension in $X$ form a Zariski open subset of the space of all sections. When $f_\bullet$ represents such a choice, we call $(\curly E_\bullet,f_\bullet)\to X$ a \emph{generic} $Q$-bundle, and in this case, the $K$-class of the degeneracy locus is independent of the maps. We will consider only this situation, and therefore are justified in omitting any decoration referring to $f_\bullet$ in our notation, e.g.\ as in the definition of Equation (\ref{eqn:dldefn}).
\end{remark}

\section[Resolution of singularities]{Resolution of singularities}\label{s:RS}

In general, the degeneracy locus $\Omega(\curly E_\bullet)$ defined by (\ref{eqn:dldefn}) is singular, though in the case of Dynkin quivers some ``worst-case scenario" results have been established. For example, it is known \cite{BZA} that over any algebraically closed field $\Omega(\curly E_\bullet)$ has at worst rational singularities when $Q$ is of type $A$, and when one assumes additionally that the field has characteristic zero the same is true for type $D$ \cite{BZD}. We work exclusively over $\C$ so the additional technical assumption that $\Omega$ have rational singularities is necessary only when $Q$ is of exceptional type (i.e.\ its underlying non-oriented graph is the Dynkin diagram for $E_6$, $E_7$, or $E_8$).

The proof of our main theorem will depend on a construction originally due to Reineke \cite{Re} to resolve the singularities, but we follow a slightly more general approach as in \cite{BuchQCDT} and adapt it specifically for $Q$-bundles. For still more details, see also \cite{qr3}.

Let $(\curly E_\bullet,f_\bullet)\to X$ be a generic $Q$-bundle. Given $i\in Q_0$ and an integer $1\leq r \leq v_i$, we construct the \emph{Grassmannization} $\Gr{v_i-r}{\curly E_i}\to X$ with tautological exact sequence $\curly S\to\curly E\to\curly Q$. Here $\curly S$ is the \emph{tautological subbundle} (whose rank is $s=v_i-r$) and $\curly Q$ is the \emph{tautological quotient bundle} (whose rank is $r$). Define $X_{i,r}(\curly E_\bullet,f_\bullet)=X_{i,r}$ to be the zero scheme $Z(\curly M_i\to\curly Q)\subset \Gr{s}{\curly E_i}$ where $\curly M_i = \Dirsum_{j\in T(i)}\curly E_j$. Observe that over $X_{i,r}\subset \Gr{s}{\curly E_i}$ we obtain an \emph{induced} $Q$-bundle $(\tilde{\curly E}_{\bullet},\tilde{f}_\bullet)$ defined by the following: 
\begin{itemize}
\item for $j\neq i$, set $\tilde{\curly E_j} = \curly E_j$,
\item set $\tilde{\curly E_i} = \curly S$,
\item if $a\in Q_1$ such that $h(a)\neq i$ and $t(a)\neq i$, then $\tilde{f}_a = f_a$,
\item if $t(a) = i$, set $\tilde{f}_a= f_a|_{\curly S}$,
\item if $h(a) = i$, set also $\tilde{f}_a=f_a$.
\end{itemize}
The last bullet is well-defined (and this is the key point) since $y\in Z(\curly M_i\to\curly Q)$ implies that in the fiber over $y$, the image of $(f_a)_y:({\curly E}_{t(a)})_y \to (\curly E_i)_y$ must lie in $\curly S_y$. Let $\rho_i^r:X_{i,r}\to X$ denote the natural map given by the composition $X_{i,r}=Z(\curly M_i,\curly Q) \includes \Gr{s}{\curly E_i} \to X$.

More generally, let $\al{i}=(i_1,\ldots,i_p)$ be a sequence of quiver vertices, and $\al{r}=(r_1,\ldots,r_p)$ a sequence of non-negative integers subject to the restriction that for each $i\in Q_0$, we have $v_i \geq \sum_{i_\ell = i}r_\ell$. We can now inductively apply the construction above to obtain a new variety $$X_{\al{i},\al{r}} = (\cdots((X_{i_1,r_1})_{i_2,r_2})\cdots)_{i_p,r_p}.$$ Let $\rho_{\al{i}}^{\al{r}}:X_{\al{i},\al{r}}\to X$ denote the natural mapping obtained from the composition $\rho_{i_1}^{r_1}\compose\cdots\compose\rho_{i_p}^{r_p}$.

Now identify each simple root $\varphi_i \in \Phi^+$ for $1\leq i \leq N$ with the standard unit vector in $\N^N$ with $1$ in position $i$ and $0$ elsewhere. For dimension vectors $\al{u},\al{w}\in\N^N$, let $$\langle \al{u},\al{w} \rangle = \sum_{i\in Q_0} u_{i}w_{i} - \sum_{a\in Q_1}u_{t(a)}w_{h(a)}$$ denote the \emph{Euler form} associated to the quiver $Q$. If $\Phi'\subset \Phi^+$ is any subset of positive roots, a partition $\Phi' = \curly I_1\union \cdots \union \curly I_\ell$ is called \emph{directed} if for every $1\leq j\leq \ell$, one has \begin{itemize} \item $\langle \alpha,\beta \rangle \geq 0$ for all $\alpha,\beta\in \curly I_j$, and \item $\langle \alpha,\beta \rangle\geq 0 \geq \langle \beta,\alpha\rangle$ whenever $i<j$ and $\alpha\in \curly I_i$, $\beta\in\curly I_j$.\end{itemize}
For Dynkin quivers a directed partition always exists \cite{Re}.

Now choose $m=(m_\varphi)_{\varphi\in\Phi^+}$, a vector of non-negative integers corresponding to the quiver cycle $\Omega_m$. Let $\Phi'\subset \Phi^+$ be a subset containing $\{\varphi \;|\; m_\varphi\neq 0\}$, and let $\Phi'=\curly I_1\union\cdots\union \curly I_\ell$ be a directed partition. For each $1\leq j \leq \ell$, compute the vector $$\sum_{\varphi\in\curly I_j}m_\varphi\varphi = (p_1^{(j)},\ldots,p_N^{(j)}) \in \N^N.$$ From this data, construct the sequence $\al{i}_j=(i_1,\ldots,i_{n})$, to be any list of the vertices $i\in Q_0$ for which $p_i^{(j)}\neq 0$, with no vertices repeated, and ordered so that for every $a\in Q_1$ the vertex $t(a)$ comes before $h(a)$. From this information, set $\al{r}_j = (p_{i_1}^{(j)},\ldots,p_{i_n}^{(j)})$. Finally, let $\al{i}$ and $\al{r}$ be the concatenated sequences $\al{i}=\al{i}_1\cdots\al{i}_\ell$ and $\al{r}=\al{r}_1\cdots\al{r}_\ell$. A pair of sequences $(\al{i},\al{r})$ constructed in this way is called a \emph{resolution pair} for $\Omega_m$.

\begin{prop}[Reineke]\label{prop:Re}
Let $Q$ be a Dynkin quiver, $\Omega_m$ a quiver cycle, and $(\al{i},\al{r})$ a resolution pair for $\Omega_m$. Then in the notation above, the natural map $\rho_{\al{i}}^{\al{r}}: X_{\al{i},\al{r}}\to X$ is a resolution of $\Omega_m(\curly E_\bullet)$; i.e.\ it has image $\Omega_m(\curly E_\bullet)$ and is a birational isomorphism onto this image.\qed
\end{prop}

The important consequence of Reineke's theorem is the following corollary.

\begin{cor}\label{cor:Re}
With $\rho_{\al{i}}^{\al{r}}$ as above, $(\rho_{\al{i}}^{\al{r}})_{*}(1) = [\curly O_{\Omega_m(\curly E_\bullet)}] \in K(X)$.\qed
\end{cor}

In the above statement $1\in K(X_{\al{i},\al{r}})$ is the class $[\curly O_{X_{\al{i},\al{r}}}]$. As we will see in Section \ref{s:PMT}, this provides an inductive recipe to give a formula for our desired $K$-class, which has been used previously by Buch e.g.\ in \cite{BuchQCDT}. However, our method of computing push-forward maps by iterated residues, which we explain in Sections \ref{s:IR} and \ref{s:ELIR}, is essentially different, and this technology produces formulas in a more compact form. For an analogous approach to this problem in the cohomological setting see \cite{qr3}.

\section[IR operations]{Iterated residue operations}\label{s:IR}

Let $f(x)$ be a rational function in the variable $x$ with coefficients in some commutative ring $R$ which has a formal Laurent series expansion in $R[[x^{\pm 1}]]$. Define the operation
\begin{equation}
\IR_{x=0,\infty}(f(x)\,dx) = \IR_{x=0}(f(x)\,dx)+\IR_{x=\infty}(f(x)\,dx),
\end{equation}
where $\IR_{x=0}(f(x)\,dx)$ is the usual residue operation from elementary complex analysis (i.e.\ take the coefficient of $x^{-1}$ in the corresponding Laurent series about $x=0$), and furthermore one recalls that $\IR_{x= \infty}(f(x)\,dx) = \IR_{x=0}(d f(\frac{1}{x}))$. The idea of using the operation $\IR_{x=0,\infty}$ in $K$-theory is due to Rim\'{a}nyi and Szenes \cite{kthom}.

More generally, let $\al{z}=\{z_1,\ldots,z_n\}$ be an alphabet of ordered commuting indeterminants and $F(\al{z})$ a rational function in these variables with coefficients in $R$ having a formal multivariate Laurent series expansion in $R[[z_1^{\pm 1},\ldots,z_n^{\pm 1}]]$. Then one defines $$\IR_{\al{z}=0,\infty}(F(\al{z})\,d\al{z}) = \IR_{z_n=0,\infty}\cdots\IR_{z_1=0,\infty}(F(\al{z})\,dz_1\cdots dz_n).$$

\begin{example}\label{ex:res1} Consider the function $g(a) = \frac{1}{(1-a/b)a}$, and the residue operation
$\IR_{a=0,\infty}(g(a)\,da)$. Using the convention that $a<<b$ (which we use throughout the sequel), we obtain that $$\IR_{a=0}(g(a)\,da) = \IR_{a=0}\left(\frac{1}{a}\left(1+\frac{a}{b}+\frac{a^2}{b^2}+\cdots\right)\,da\right) = 1.$$ On the other hand, $$-\frac{1}{a^2}g(1/a) = b\left(\frac{1}{1-ab}\right)$$ and so $\IR_{a=\infty}(g(a)\,da) = 0$. Thus $\IR_{a=0,\infty}(g(a)\,da) = 1$. However, it is more convenient to do the calculation by using the fact that for any meromorphic differential form the sum of all residues (including the point at infinity) is zero. Since the only other pole of $g$ occurs at $a=b$, we see easily that $$\IR_{a=0,\infty}(g(a)\,da) = -\IR_{a=b}\left( \frac{da}{(1-a/b)a}\right) = 1. $$
\end{example}

\begin{example}\label{ex:res2} Consider the meromorphic differential form $$F(z_1,z_2) = \frac{(1-\frac{\beta_1}{z_2})(1-\frac{\beta_2}{z_2})(1-\frac{z_2}{z_1})}{(1-\frac{z_1}{\alpha_1})(1-\frac{z_2}{\alpha_1})(1-\frac{z_1}{\alpha_2})(1-\frac{z_2}{\alpha_2})z_1 z_2}dz_1 dz_2.$$ Functions of this type will occur often in our analysis, where the result of the operation $\Res_{\al{z}=0,\infty}(F)$ is a certain (Laurent) polynomial in the variables $\alpha_i$ and $\beta_j$, separately symmetric in each. We begin by factoring $F=F_1 F_2$, where
\begin{align*} 
F_1 & = \frac{(1-\frac{z_2}{z_1})}{(1-\frac{z_1}{\alpha_1})(1-\frac{z_1}{\alpha_2})z_1}dz_1  & \text{and} &
 &F_2 & = \frac{(1-\frac{\beta_1}{z_2})(1-\frac{\beta_2}{z_2})}{(1-\frac{z_2}{\alpha_1})(1-\frac{z_2}{\alpha_2})z_2 }dz_2.
\end{align*}
We first use the residue theorem as in the previous example to write that $$\IR_{z_1=0,\infty}(F) =-\left(\IR_{z_1=\alpha_1}(F)+\IR_{z_1=\alpha_2}(F)\right),$$ and we compute that 
\begin{align*}
-\IR_{z_1=\alpha_1}(F) & =-F_2 \left(\IR_{z_1=\alpha_1}(F_1) \right)= F_2 \left(\frac{(1-\frac{z_2}{\alpha_1})}{(1-\frac{\alpha_1}{\alpha_2})}\right) = F'\\
-\IR_{z_1=\alpha_2}(F) & =-F_2 \left(\IR_{z_1=\alpha_2}(F_1) \right) = F_2 \left(\frac{(1-\frac{z_2}{\alpha_2})}{(1-\frac{\alpha_2}{\alpha_1})}\right) = F''
\end{align*}
It is not difficult to see that $\IR_{z_2=\alpha_1}(F')=\IR_{z_2=\alpha_2}(F'')=0$, so it remains only to compute
\begin{align*}
\IR_{\al{z}=0,\infty}(F)
&= -\IR_{z_2=\alpha_2}(F')-\IR_{z_2=\alpha_1}(F'')\\
&= \frac{(1-\frac{\beta_1}{\alpha_2})(1-\frac{\beta_2}{\alpha_2})}{(1-\frac{\alpha_1}{\alpha_2})} + \frac{(1-\frac{\beta_1}{\alpha_1})(1-\frac{\beta_2}{\alpha_1})}{(1-\frac{\alpha_2}{\alpha_1})}
= 1-\frac{\beta_1\beta_2}{\alpha_1\alpha_2}.
\end{align*}
The last line above bears resemblance to a Berline-Vergne-Atiyah-Bott type formula for equivariant localization, adapted for $K$-theory. This is not accidental, a connection which we explain in Section \ref{s:ELIR}.
\end{example}

\section[Main theorem]{The main theorem}\label{s:MT}

Choose an element $m=(m_\varphi) \in \N^{\Phi^+}$ corresponding to the $\al{G}$-orbit closure $\Omega_m \subset V$, having only rational singularities. Let $\al{i}=(i_1,\ldots,i_p)$ and $\al{r}=(r_1,\ldots,r_p)$ be a resolution pair for $\Omega_m$. Let $(\curly E_{\bullet},f_\bullet) \to X$ be a generic $Q$-bundle over the smooth complex projective base variety $X$. For each $k \in \{1,\ldots,p\}$ define alphabets of ordered commuting variables
$$\al{z}_k =\{z_{k1},\ldots,z_{kr_k}\}$$
and the \emph{discriminant} factors
$$\Delta(\al{z}_k) = \prod_{1\leq i < j \leq r_k} \left( 1 - \frac{z_{kj}}{z_{ki}} \right).$$
For each $i\in Q_0$, recall the definition of the set $T(i)$ from Equation (\ref{eqn:T(i)}), and define the alphabets of commuting variables
\begin{align*}
\mathbb E_i &= \{\epsilon_{i1},\ldots,\epsilon_{iv_i}\}, & \mathbb M_i & = \Union_{j\in T(i)} \mathbb E_j
\end{align*}
where the degree $d$ elementary symmetric function $e_d(\mathbb E_i) = e_d(\epsilon_{i1},\ldots,\epsilon_{iv_i})= [\wedge^d (\curly E_i)] \in K(X)$. Consequently, we conclude that $e_d(\epsilon_{i1}^{-1},\ldots,\epsilon_{iv_i}^{-1}) = [\wedge^d (\curly E_i\dual)]$. Henceforth, we will call such a set of formal commuting variables \emph{Grothendieck roots} of $\curly E_i$. Finally, for each $k\in\{1,\ldots,p\}$ define 
\begin{itemize}\item the \emph{residue factors}
\begin{equation}
R_k = \prod_{y\in\al{z_k}}\frac{\prod_{x\in \mathbb{M}_k}(1-x y)}{\prod_{x\in \mathbb{E}_k}(1- x y)} \nonumber
\end{equation}
\item the \emph{interference factors}
\begin{equation}
I_k = \prod_{y \in \al{z_k}} \frac{\displaystyle \mathop{\prod_{\ell<k:i_\ell=i_k}}_{x\in \al{z_\ell}}\left( 1- \frac{y}{x}\right)}{\displaystyle \mathop{\prod_{\ell<k:i_\ell\in T(i_k)}}_{x\in \al{z_\ell}}\left( 1- \frac{y}{x}\right)}\nonumber
\end{equation}
\item and the \emph{differential factors}
\begin{equation}
D_k = \Delta(\al{z}_k) \cdot \dlog(\al{z}_k) = \Delta(\al{z}_k)\prod_{i=1}^{r_k}\frac{dz_{ki}}{z_{ki}}.\nonumber
\end{equation}
\end{itemize}

\begin{thm}\label{thm:main}
With the notations above, the class $[\curly O_{\Omega_m(\curly E_\bullet)}] \in K(X)$ is given by the iterated residue
\begin{equation}\label{eqn:main}
\IR_{\al{z}_1=0,\infty}\cdots\IR_{\al{z}_p=0,\infty}
\left(
\prod_{k=1}^p R_k I_k D_k \right).
\end{equation}
\end{thm}

\begin{example}\label{ex:inA3MT}
Consider the ``inbound $A_3$" quiver $\{1\rightarrow 2 \leftarrow 3\}$. Let $\varphi_1$, $\varphi_2$, and $\varphi_3$ be the corresponding simple roots so that the positive roots of the underlying root system can be represented by $\varphi_{ij} = \sum_{i\leq \ell \leq j} \varphi_\ell$ for $1\leq i \leq j \leq 3$. Consider now the orbit closure $\Omega_m \subset V=\Hom(E_1,E_2)\dirsum\Hom(E_3,E_2)$  corresponding to $m_{11}=m_{23}=0$, but all other $m_{ij}=1$ so that the resulting dimension vector is $\al{v}=(2,3,2)$. Set $\Phi'=\{\varphi_{12},\varphi_{13},\varphi_{22},\varphi_{33}\}$ and choose the directed partition $$\Phi'=\{\varphi_{22}\} \union \{\varphi_{12},\varphi_{13}\} \union \{\varphi_{33}\}$$ with corresponding resolution pair $\al{i} = (2,1,3,2,3)$ and $\al{r}=(1,2,1,2,1)$. Let $\curly E_\bullet\to X$ be a generic $Q$-bundle. Set \begin{align*} \mathbb E_1 &=\{\alpha_1,\alpha_2\}, &\mathbb E_2 &= \{\beta_1,\beta_2,\beta_3\}, & \mathbb E_3 & = \{\gamma_1,\gamma_2\} \end{align*} to be the Grothendieck roots of $\curly E_1$, $\curly E_2$, and $\curly E_3$ respectively. In particular, this means that $\mathbb M_1 = \mathbb M_3 = \{\}$ while $\mathbb M_2 = \{\alpha_1,\alpha_2,\gamma_1,\gamma_2\}$. Following the recipe of the theorem and equation (\ref{eqn:main}) we form the alphabets $\al{z_k}$ for $1\leq k \leq 5$ which we rename as
\begin{align*}
\al{z}_1 &=\{v\} & \al{z}_2 &=\{w_1,w_2\} & \al{z}_3 &=\{x\} & \al{z}_4 &=\{y_1,y_2\} & \al{z}_5 &=\{z\}
\end{align*}
and construct the differential form
\begin{equation}\label{eqn:inA3form}
\frac{\displaystyle
	\prod_{s\in \mathbb M_2,t\in \al{z}_1\union\al{z}_4} \left(1 - s t\right)
	} 
     {\displaystyle
     \mathop{\prod_{s\in \mathbb E_2}}_{t\in \al{z}_1\union\al{z}_4}(1- s t)
     \mathop{\prod_{s\in \mathbb E_1}}_{t\in \al{z}_2}(1-s t)
     \mathop{\prod_{s\in \mathbb E_3}}_{t\in \al{z}_3\union\al{z}_5}(1- s t)}
     \frac{\displaystyle \left(1-\frac{z}{x}\right)\prod_{i=1}^2\left(1-\frac{y_i}{v}\right)}
     	{\displaystyle \mathop{\prod_{s\in \al{z}_4}}_{t\in \al{z}_2\union\al{z}_3} \left(1-\frac{s}{t}\right)}
	\prod_{k=1}^5 D_k
\end{equation}
and a calculation in \emph{Mathematica} shows that the result of applying the iterated residue operation $\IR_{\al{z}_1=0,\infty}\cdots\IR_{\al{z}_5=0,\infty}$ to the form above gives
\begin{equation}\label{eqn:kclass}
\begin{array}{rl}
[\curly O_{\Omega_m(\curly E_\bullet)}] = & 
1
-\frac{\alpha _1 \alpha _2 \gamma _1^2 \gamma _2^2}{\beta _1^2 \beta _2^2 \beta _3^2}
+\frac{\alpha _1 \alpha _2 \gamma _1 \gamma _2}{\beta _1 \beta _2 \beta _3^2}+\frac{\alpha _1 \alpha _2 \gamma _1 \gamma _2}{\beta _1 \beta _2^2 \beta _3}+\frac{\alpha _1 \alpha _2 \gamma _1 \gamma _2}{\beta _1^2 \beta _2 \beta _3}
\\ 
& -\frac{\gamma _1 \gamma _2}{\beta _1 \beta _2}-\frac{\gamma _1 \gamma _2}{\beta _1 \beta _3}-\frac{\gamma _1 \gamma _2}{\beta _2 \beta _3}
-\frac{\alpha _1 \alpha _2 \gamma _1}{\beta _1 \beta _2 \beta _3}-\frac{\alpha _1 \alpha _2 \gamma _2}{\beta _1 \beta _2 \beta _3}+\frac{\gamma _1^2 \gamma _2}{\beta _1 \beta _2 \beta _3}+\frac{\gamma _1 \gamma _2^2}{\beta _1 \beta _2 \beta _3}
.
\end{array}
\end{equation}
Following Buch's combinatorial description of the inbound $A_3$ case (cf.\ \cite[Section~7.1]{BuchQCDT}) one obtains in terms of double stable Grothendieck polynomials that \begin{equation}\label{eqn:Gclass} [\curly O_{\Omega(\curly E_\bullet)}] = G_{21}(\curly E_2 - \curly M_2)+G_{2}(\curly E_2-\curly M_2)G_{1}(\curly E_1) -G_{21}(\curly E_2-\curly M_2)G_{1}(\curly E_1)\end{equation} which one can check agrees with equation (\ref{eqn:kclass}) once expanded (n.b.\ in the expression above the subscript ``$21$" is the partition whose Young diagram has two rows, the first with two boxes, the second with one box). The leading term above (see \cite[Corollary~4.5]{BuchQCDT}) is given by $s_{21}(\curly E_2-\curly M_2) + s_{2}(\curly E_2-\curly M_2)s_{1}(\curly E_1)$ which agrees with the result of \cite[Section~6.2]{qr3}.

We wish also to compare our result directly to the cohomological iterated residue formula of Rim\'{a}nyi, see \cite{qr3}. From the $K$-class $[\curly O_{\Omega_m(\curly E_\bullet)}]$ one obtains the cohomology class $[\Omega_m(\curly E_\bullet)]$ by the following method, which we explain in general.

Let $\curly E_1,\ldots,\curly E_n$ be vector bundles over $X$ with ranks $e_1,\ldots,e_n$ respectively, and $\mathbb E_1=\{\epsilon_{11},\ldots,\epsilon_{1e_1}\} , \ldots , \mathbb E_n = \{\epsilon_{n1},\ldots,\epsilon_{ne_n}\}$ respective sets of Grothendieck roots. If $f(\epsilon_{ij})$ is a Laurent polynomial, separately symmetric in each set of variables $\mathbb E_i$, then $f$ represents a well-defined element in $K(X)$, and for such a class replace each $\epsilon_{ij}$ with the exponential $\exp({\epsilon_{ij}\xi})$. Then a class in $H^{*}(X)$ is given by taking the lowest degree nonzero term in the Taylor expansion (with respect to $\xi$ about zero) of $f(\exp({\epsilon_{ij}\xi}))$ where, once in the cohomological setting, the variables $\epsilon_{ij}$ are interpreted as \emph{Chern roots} of the corresponding bundles. In particular, applying this process to the class $[\curly O_{\Omega(\curly E_\bullet)}]$ yields the class $[\Omega(\curly E_\bullet)] \in H^{*}(X)$. This is actually the leading term of the \emph{Chern character} $K(X)\to H^{*}(X)$. For more details, see \cite[Section~4]{BuchQCDT}.

Applying the algorithm above to the Laurent polynomial (\ref{eqn:kclass}) gives that the corresponding class in $H^{*}(X)$ must be
\begin{eqnarray*}
[\Omega_m(\curly E_\bullet)] &=&   2 \beta_1 \beta_2 \beta_3+\beta_1^2 \beta_2+\beta_1 \beta_2^2+\beta_1^2 \beta_3+\beta_2^2 \beta_3+\beta_1 \beta_3^2+\beta_2 \beta_3^2  \\ 
& &-\alpha_1 \beta_1 \beta_2 -\alpha_2 \beta_1 \beta_2-\alpha_1 \beta_2 \beta_3-\alpha_2 \beta_2 \beta_3-\alpha_1 \beta_1 \beta_3-\alpha_2 \beta_1 \beta_3  \\ 
& &-2\left( \beta_1 \beta_2 \gamma_1 + \beta_1 \beta_2 \gamma_2 + \beta_2 \beta_3 \gamma_1 + \beta_2 \beta_3 \gamma_2 + \beta_1 \beta_3 \gamma_1 + \beta_1 \beta_3 \gamma_2 \right)  \\ 
& &-\beta_1^2 \gamma_1-\beta_1^2 \gamma_2-\beta_2^2 \gamma_1-\beta_2^2 \gamma_2 -\beta_3^2 \gamma_1-\beta_3^2 \gamma_2 \\
& &+\beta_1 \gamma_1^2+\beta_2 \gamma_1^2+\beta_3 \gamma_1^2+\beta_1 \gamma_2^2+\beta_2 \gamma_2^2+\beta_3 \gamma_2^2  \\ 
& &+2 \left( \beta_1 \gamma_1 \gamma_2 + \beta_2 \gamma_1 \gamma_2 + \beta_3 \gamma_1 \gamma_2 \right) -\gamma_1^2 \gamma_2-\gamma_1 \gamma_2^2  \\
& &-\alpha_1 \gamma_1^2-\alpha_2 \gamma_1^2-\alpha_1 \gamma_2^2-\alpha_2 \gamma_2^2-\alpha_1 \gamma_1 \gamma_2-\alpha_2 \gamma_1 \gamma_2  \\ 
& &+\alpha_1 \beta_1 \gamma_1+\alpha_2 \beta_1 \gamma_1+\alpha_1 \beta_2 \gamma_1+\alpha_2 \beta_2 \gamma_1+\alpha_1 \beta_3 \gamma_1+\alpha_2 \beta_3 \gamma_1  \\
& &+\alpha_1 \beta_1 \gamma_2+\alpha_2 \beta_1 \gamma_2+\alpha_1 \beta_2 \gamma_2+\alpha_2 \beta_2 \gamma_2+\alpha_1 \beta_3 \gamma_2+\alpha_2 \beta_3 \gamma_2 
\end{eqnarray*}
where the variables $\{\alpha_i\}$, $\{\beta_i\}$, and $\{\gamma_i\}$ are now interpreted as the Chern roots of $\curly E_1$, $\curly E_2$, and $\curly E_3$ respectively. If one sets $A_i = c_i(\curly E_1)$, $B_i = c_i(\curly E_2)$, and $C_i = c_i(\curly E_3)$ to be the corresponding Chern classes the expression above becomes
\begin{equation}\label{eqn:KtoHclass}
[\Omega_m] =
\left(B_1 - A_1\right) \left(B_2+C_1^2\right)-C_1\left(B_1^2+C_2\right) +A_1 \left(B_1 C_1+C_2\right) -B_3.
\end{equation}
In \cite[Equation~(9)]{qr3}, this class is computed to be
\begin{equation}\label{eqn:RRclass} -c_3(\curly M_2\dual-\curly E_2\dual) + c_2(\curly M_2\dual-\curly E_2\dual)c_1(\curly M_2\dual-\curly E_2\dual) + c_2(\curly M_2\dual-\curly E_2\dual)c_1(-\curly E_1\dual)
\end{equation}
where the \emph{relative Chern classes} $c_n(\curly V\dual-\curly W\dual)$ are defined by the formal expression 
$$\sum_{n\geq 0} c_n(\curly V\dual-\curly W\dual)\xi^n =
\frac{\sum_{k\geq 0}c_k(\curly V)(-\xi)^k}{\sum_{\ell\geq0}c_\ell(\curly W)(-\xi)^\ell}$$
for bundles $\curly V$ and $\curly W$ with respective Chern classes $c_k(\curly V)$ and $c_\ell(\curly W)$. Using the Chern classes $A_i$, $B_i$, and $C_i$ as above, one substitutes into the expression (\ref{eqn:RRclass}) to obtain 
\begin{align*}
[\Omega_m]=&-\left[(B_1^3+B_3-2B_1B_2)-(B_1^2-B_2)(A_1+C_1) \right. \\ 
&\left.\hphantom{(B_1^3+B_3-2B_1B_2)-(B_1^2)}+B_1(A_2+A_1C_1+C_2)-(A_1C_2+A_2C_2)\right]\\
&+[(B_1^2-B_2)-B_1(A_1+C_1)+(A_2+A_1C_1+C_2)][B_1-(A_1+C_1)] \\
&+[(B_1^2-B_2)-B_1(A_1+C_1)+(A_2+A_1C_1+C_2)]A_1
\end{align*}
and a little high-school algebra shows that this is identical to (\ref{eqn:KtoHclass}).
\end{example}

\begin{remark}
The leading term of the class (\ref{eqn:Gclass}) is, according to Buch, called $s_{21}(\curly E_2-\curly M_2) + s_{2}(\curly E_2-\curly M_2)s_{1}(\curly E_1)$. In \cite{qr3}, the same Schur functions are instead evaluated on $\curly M_i\dual - \curly E_i\dual$, but both authors' notations are interpreted to mean $$s_\lambda = \det(h_{\lambda_i+j-i})$$ where the $h_\ell$ are the appropriate relative Chern classes defined above.
\end{remark}

\section[Localization]{Equivariant localization and iterated residues}\label{s:ELIR}

Let $X$ be a smooth complex projective variety and $\curly A \to X$ a vector bundle of rank $n$. Choose an integer $1\leq k \leq n$ and set $q=n-k$. The integers $n$, $k$, and $q$ will be fixed throughout the section. Form the Grassmannization of $\curly A$ over $X$, $\pi:\Gr{k}{\curly A}\to X$, with tautological exact sequence of vector bundles $\curly S\to \curly A \to \curly Q$ over $\Gr{k}{\curly A}$. By convention, we suppress the notation of pullback bundles. The following diagram is useful to keep in mind:
\begin{equation}
\begin{diagram}[2em]
\curly A	&	&\curly S	&\rTo	&\curly A			&\rTo	&\curly Q	\\
\dTo		&	&			&\rdTo	&\dTo				&\ldTo	&			\\
X			&	&\lTo_{\pi}	&		&\Gr{k}{\curly A}	&		&
\end{diagram}\nonumber
\end{equation}
Let $\{\sigma_1,\ldots,\sigma_k\}$ and $\{\omega_1,\ldots,\omega_q\}$ be sets of Grothendieck roots for $\curly S$ and $\curly Q$ respectively. Set $R=K(X)$ and let $f$ be a Laurent polynomial in $R[\sigma_i^{\pm 1};\omega_j^{\pm 1}]$ separately symmetric in the $\sigma$ and $\omega$ variables, (where $1\leq i\leq k$ and $1\leq j \leq q$). The symmetry of  $f$ implies that it represents a $K$-class in $K(\Gr{k}{\curly A})$. The purpose of this section is to give an explanation of the push-forward map $\pi_{*}:K(\Gr{k}{\curly A}) \to K(X)$ applied to $f$.

Many formulas for $\pi_{*}$ exist in the literature. For example, Buch has given a formula in terms of stable Grothendieck polynomials and the combinatorics of integer sequences in \cite[Theorem~7.3]{BuchGCQV}. We will utilize the method of equivariant localization.  The following formula is well-known to experts, deeply embedded in the folklore of the subject and, as such, a single (or original) reference is unknown to the author. Following the advice of \cite{FSM}, we refer the reader to various sources, namely \cite{EqKT} and \cite{CG}.

\begin{prop}\label{prop:locpf}
Let $\{\alpha_1,\ldots,\alpha_n\}$ be Grothendieck roots for $\curly A$ and set $[n]=\{1,\ldots,n\}$. Let $[n,k]$ denote the set of all $k$-element subsets of $[n]$, and for any subset $J=\{j_1,\ldots,j_r\}\subset [n]$, let $\alpha_J$ denote the collection of variables $\{\alpha_{j_1},\ldots,\alpha_{j_r}\}$. With the notation above $\pi_{*}$ acts by
\begin{equation}
f(\sigma_1,\ldots,\sigma_k ; \omega_1,\ldots,\omega_q) \mapsto \sum_{I\in [n,k]} \frac{f(\alpha_I ; \alpha_{\bar{I}})}{\prod_{i\in I,j\in \bar{I}} (1-\frac{\alpha_i}{\alpha_j} )}\nonumber
\end{equation}
where $\bar{I}$ denotes the complement $[n]\setminus I$.
\end{prop}

\begin{example}\label{ex:locpf}
Suppose that $\curly A$ and $\curly B$ are both vector bundles of rank $2$ and let $\{\alpha_1,\alpha_2\}$ be as above. Let $\{\beta_1,\beta_2\}$ be Grothendieck roots of $\curly B$. Form the Grassmannization $\Gr{1}{\curly A} = \PP(\curly A)$ and consider the class $$f(\sigma,\omega) = \left(1-\frac{\beta_1}{\omega}\right) \left(1-\frac{\beta_2}{\omega}\right) \in K(\PP(\curly A)).$$ The expert will recognize this expression as the $K$-class associated to the structure sheaf of the subvariety in $\PP(\curly A)$ defined by the vanishing of a generic section $\mathbb{P}(\curly A) \to \Hom(\curly B,\curly Q)$. In any event, applying Proposition \ref{prop:locpf} gives that
$$\pi_{*}(f(\sigma,\omega)) = \frac{(1-\frac{\beta_1}{\alpha_2})(1-\frac{\beta_2}{\alpha_2})}{(1-\frac{\alpha_1}{\alpha_2})} + \frac{(1-\frac{\beta_1}{\alpha_1})(1-\frac{\beta_2}{\alpha_1})}{(1-\frac{\alpha_2}{\alpha_1})}$$
an expression which we concluded was equal to $1-\frac{\beta_1\beta_2}{\alpha_1\alpha_2}$ in Example \ref{ex:res2}. In comparison to Buch's formula (cf. \cite[Theorem~7.3]{BuchGCQV}) we have set $f=G_2(\curly Q-\curly B)$ and obtained that $\pi_{*}(f) = G_1(\curly A-\curly B)$.
\end{example}

Observe that in general, the expression obtained from applying Proposition \ref{prop:locpf} has many terms (the binomial coefficient ${n \choose k}$ to be precise) and by this measure is quite complicated. Hence we seek to encode the expression in a more compact form, and this is accomplished by the following proposition, which is just a clever rewriting of the localization formula, pointed out to the author by Rim\'{a}nyi in correspondence with Szenes.

\begin{prop}\label{prop:IRpf}
Let $\al{z}=\{z_1,\ldots,z_n\}$ be an alphabet of ordered, commuting variables. If $f$ has no poles in $R=K(X)$ (aside from zero and the point at infinity), then in the setting of Proposition \ref{prop:locpf} one has that $\pi_{*}$ acts by
\begin{equation}
f(\sigma_1,\ldots,\sigma_k;\omega_1,\ldots,\omega_q) \mapsto \IR_{\al{z}=0,\infty}\left( f(\al{z}) \frac{\prod_{1\leq i < j\leq n} \left( 1 - \frac{z_j}{z_i}\right) }{ \prod_{i,j=1}^{n} \left( 1-\frac{z_i}{\alpha_j} \right) } \dlog\al{z} \right)\nonumber
\end{equation}
where $\dlog\al{z} = \prod_{i=1}^n d\log(z_i) = \prod_{i=1}^n \frac{dz_i}{z_i}$.
\end{prop}

\begin{proof}
The proof is a formal application of the fact that the sum of the residues at all poles (including infinity) vanishes. We give an example of this phenomenon below, and the general proof is completely analogous, only requiring more notation and paper. We leave the details to the reader, but for a similar proof in the case of equivariant localization and proper push-forward in cohomology see \cite{Z}.
\end{proof}

If the class represented by $f$ depends only on the variables $\sigma_i$, then the expression above can be dramatically simplified--namely one needs to utilize only the variables $z_i$ for $1\leq i \leq k$.

\begin{cor}\label{cor:IRpfSonly}
If $f=f(\sigma_1,\ldots,\sigma_k)$ depends only on the Grothendieck roots of $\curly S$, then set $\al{z} = \{z_1,\ldots,z_k\}$ and $\pi_{*}$ acts by
\begin{equation}
f(\sigma_1,\ldots,\sigma_k) \mapsto \IR_{\al{z}=0,\infty} \left( f(\al{z}) \frac{  \prod_{1\leq i < j\leq k} \left( 1 - \frac{z_j}{z_i}\right) }{ \displaystyle \prod_{1\leq i\leq k,1\leq j \leq n} \left( 1-\frac{z_i}{\alpha_j} \right) } \dlog\al{z} \right)\nonumber
\end{equation}
\end{cor}

\begin{proof}
We will prove the result in the case $n=2$ and $s=q=1$; the general case is analogous. Let $f(\sigma)$ represent a class in $K(\Gr{s}{\curly A})$. Proposition \ref{prop:IRpf} implies that $\pi_{*}(f)$ is
$$\IR_{\al{z}=0,\infty} \left( f(z_1) \frac{\left(1-\frac{z_2}{z_1}\right)\,\dlog\al{z}}{\prod_{i,j=1}^2 \left(1- \frac{z_i}{\alpha_j} \right)}\right).$$
Taking the ``finite" residues of $z_1=\alpha_1$ and $z_1=\alpha_2$, we obtain that the above is equal to
$$\IR_{z_2=0,\infty} \left(
\frac{f(\alpha_1)\cancel{\left(1-\frac{z_2}{\alpha_1}\right)}\,dz_2}{\cancel{\left(1-\frac{z_2}{\alpha_1}\right)}\left(1-\frac{\alpha_1}{\alpha_2}\right)\left(1-\frac{z_2}{\alpha_2}\right)z_2}
+
\frac{f(\alpha_2)\cancel{\left(1-\frac{z_2}{\alpha_2}\right)}\,dz_2}{\cancel{\left(1-\frac{z_2}{\alpha_2}\right)}\left(1-\frac{\alpha_2}{\alpha_1}\right)\left(1-\frac{z_2}{\alpha_1}\right)z_2}
\right).$$
In both terms of the expression above, the only part which depends on $z_2$ has the form $\frac{1}{(1-z_2/\alpha_i)z_2}$ and Example \ref{ex:res1} implies that residues of this type always evaluate to $1$. Observe then, that the expression above is equivalent to what we would have obtained by removing all the factors involving $z_2$ at the beginning.
\end{proof}

One can obtain a similar expression for classes depending only on the variables $\omega_j$ which requires only $n-k=q$ residue variables.

\begin{cor}\label{cor:IRpfQonly}
If $f=f(\omega_1,\ldots,\omega_q)$ depends only on the Grothendieck roots of $\curly Q$, then set $\al{z} = \{z_1,\ldots,z_q\}$ and $\pi_{*}$ acts by
\begin{equation}
f(\omega_1,\ldots,\omega_q) \mapsto \IR_{\al{z}=0,\infty} \left( f(z_1^{-1},\ldots,z_q^{-1}) \frac{  \prod_{1\leq i < j\leq k} \left( 1 - \frac{z_j}{z_i}\right) }{ \displaystyle \prod_{1\leq i\leq q,1\leq j \leq n} \left( 1-\alpha_j z_i \right) } \dlog\al{z} \right)\nonumber
\end{equation}
\end{cor}

\begin{proof}
We use the fact that $\Gr{s}{\curly A}$ is homeomorphic to the Grassmannian fibration $\Gr{q}{\curly A\dual}$, over which lies the tautological exact sequence $\curly Q\dual \to \curly A\dual \to \curly S\dual$. We are now in a situation to apply the previous corollary, once we recognize that for any bundle $\curly B$, if $\{\beta_i\}_{1\leq i\leq \rk{\curly B}}$ is a set of Grothendieck roots, then the corresponding Grothendieck roots of $\curly B\dual$ are supplied by $\{\beta_i^{-1}\}_{1\leq i \leq \rk{\curly B}}$.
\end{proof}

\section[Main theorem: proof]{Proof of the main theorem}\label{s:PMT}

In this section we prove Theorem \ref{thm:main} and will use the notation of Section \ref{s:MT} except where otherwise specified. We will need the language and notation of Reineke's construction, which is detailed in Section \ref{s:RS}. We introduce also the following notation. If $\mathbb A = \{a_1,\ldots,a_n\}$ and $\mathbb B = \{b_1,\ldots,b_m\}$ then we write
\begin{align*}
	\left(1-\frac{\mathbb A}{\mathbb B}\right) 
&= \mathop{\prod_{1\leq i \leq n}}_{1\leq j \leq m}\left(1-\frac{a_i}{b_j}\right)
&\left(1-{\mathbb A}{\mathbb B}\right) 
&= \mathop{\prod_{1\leq i \leq n}}_{1\leq j \leq m}\left(1-{a_i}{b_j}\right).
\intertext{In the special case that $\mathbb A$ and $\mathbb B$ are respective sets of Grothendieck roots of vector bundles $\curly A$ and $\curly B$, we will write $\curly A_\bullet = \mathbb A$ and $\curly B_\bullet = \mathbb B$ above. We can also mix these notations and write e.g.}
\left(1-\frac{\curly A_\bullet}{\mathbb B}\right) 
&= \mathop{\prod_{1\leq i \leq n}}_{1\leq j \leq m}\left(1-\frac{a_i}{b_j}\right)
&\left(1-\frac{\mathbb A}{\curly B_\bullet}\right) 
&= \mathop{\prod_{1\leq i \leq n}}_{1\leq j \leq m}\left(1-\frac{a_i}{b_j}\right)\\
\left(1-{\curly A_\bullet}{\mathbb B}\right) 
&= \mathop{\prod_{1\leq i \leq n}}_{1\leq j \leq m}\left(1-{a_i}{b_j}\right)
&\left(1-{\mathbb A}{\curly B_\bullet}\right) 
&= \mathop{\prod_{1\leq i \leq n}}_{1\leq j \leq m}\left(1-{a_i}{b_j}\right)
\end{align*}
if only $\mathbb A$ corresponds to a set of Grothendieck roots and $\mathbb B$ represents a set of some other formal variables (as on the left) or vice versa (as on the right). This is not to be confused with the notation $\curly E_\bullet \to X$ used to denote a $Q$-bundle. The context should always make clear the intended meaning of the ``bullet" symbol as a subscript to calligraphic letters.

We will prove Theorem \ref{thm:main} by iteratively understanding the sequence of maps $\rho_{i_k}^{r_k}$ in the Reineke resolution, which break up into a natural inclusion followed by a natural projection from a Grassmannization (cf.\ Section \ref{s:RS}). Our first step is the following lemma, which provides a formula for the natural inclusion.

\begin{lem}\label{lem:iota}
Let $X$ be a smooth base variety and $\curly M\to \curly E$ a map of vector bundles over $X$. Let $0\leq s\leq \rk(\curly E)$ and form the Grassmannization $\pi:\Gr{s}{\curly E}\to X$ with tautological exact sequence $\curly S\to \curly E \to \curly Q$. Set $Z=Z(\curly M\to \curly Q)\subset \Gr{s}{\curly E}$ and let $\iota:Z\includes\Gr{s}{\curly E}$ denote the natural inclusion. If $f\in K(Z)$ is a class expressed entirely in terms of bundles pulled back from $\Gr{s}{\curly E}$ then $\iota_{*}:K(Z)\to K(\Gr{s}{\curly E})$ acts on $f$ by
$$
f\mapsto f\cdot \left(1-\frac{\curly M_\bullet}{\curly Q_\bullet}\right).
$$
\end{lem}

\begin{proof}
Set $r=\rk(\curly Q) = \rk(\curly E)-s$ and $m=\rk(\curly M)$. Because of our assumption on $f$ we know that $\iota_{*}(f) = \iota_{*}(\iota^{*}(f))$, and therefore the adjunction formula implies that $\iota_{*}(f) = f\cdot \iota_{*}(1)$. The image of $\iota_{*}(1)$ is exactly the class $[\curly O_{Z(\curly M\to\curly Q)}] \in K(\Gr{s}{\curly E})$ which is given by the $K$-theoretic Giambelli-Thom-Porteous theorem \cite[Theorem~2.3]{BuchGCQV}. Explicitly, $$\iota_{*}(1) = G_R(\curly Q - \curly M)$$ where $G_R$ denotes the double stable Grothendieck polynomial associated to the rectangular partition $R=(m)^{r}$, i.e.\ the partition whose Young diagram has $r$ rows each containing $m$ boxes. The result of evaluating $G_R$ on the bundles in question is given, e.g.\ by \cite[Equation~(7.1)]{BuchKLR} $$G_R(\curly Q - \curly M) = G_R(x_1,\ldots,x_{r};y_1,\ldots,y_{m}) = \mathop{\prod_{1\leq i \leq r}}_{1\leq j \leq m}(x_i+y_j-x_iy_j)$$ with the specializations $x_i=1-\omega_i^{-1}$ and $y_j=1-\mu_j$, where $\curly Q_\bullet = \{\omega_i\}_{i=1}^{r}$ and $\curly M_\bullet=\{\mu_j\}_{j=1}^{m}$ denote the respective Grothendieck roots. The result of this substitution is exactly the statement of the lemma.
\end{proof}

For the Dynkin quiver $Q$, smooth complex projective variety $X$, and quiver cycle $\Omega$, let $\curly E_\bullet\to X$ be a generic $Q$-bundle and $\al{i}=(i_1,\ldots,i_p)$, $\al{r}=(r_1,\ldots,r_p)$ be a resolution pair for $\Omega$. We will show that at each step in the Reineke resolution, the result can be written as an iterated residue entirely in terms of residue variables (i.e.\ the alphabets $\al{z}_k$) and Grothendieck roots only of the bundles $\curly E_i$ or the tautological quotient bundles constructed at previous steps. Moreover, the form of this result is arranged in such a way to evidently produce the formula of the main theorem.

By Corollary \ref{cor:Re} we must begin with the image of $(\rho_{i_p}^{r_p})_{*}(1)$. Set $i=i_p\in Q_0$ and $\curly A=\curly E_i$. Write $T(i) = \{t_1,\ldots,t_\ell\}\subset Q_0$ and denote $\curly E_{t_j}=\curly B_j$. Recall that whenever $j\in Q_0$ appears in the Reineke resolution sequence $\al{i}$, it is subsequently replaced with a tautological subbundle. For any bundle $\curly F$ and Grassmannization $\Gr{s}{\curly F}$, we will denote the tautological subbundle by $\curly S\curly F$. If this is done multiple times, we let $\curly S^n \curly F$ denote the tautological subbundle over $\Gr{s'}{\curly S^{n-1} \curly F}$. Similarly, we denote the tautological quotient over $\Gr{s}{\curly F}$ by $\curly Q \curly F$.

Suppose that the vertex $i\in Q_0$ appears $n$ times in $\al{i}$ and moreover that each tail vertex $t_j$ appears $n_j$ times. Set
\begin{align*}
Y &=(\cdots(X)_{i_1,r_1}\cdots)_{i_{p-1},r_{p-1}}, &
\curly M &= \Dirsum_{j=1}^{\ell}\curly S^{n_j}\curly B_j, &
Z &=Z(\curly M\to\curly Q \curly S^{n-1}\curly A).
\end{align*}
Then the composition $\rho_{i_p}^{r_p}=\pi_p\compose\iota_p$ is depicted diagrammatically below
$$
\begin{diagram}[width=1em,height=2em]
\curly M	& 		& \rTo	& 		&\curly S^{n-1}\curly A		&		&
\curly S^{n}\curly \curly A	& \rTo	& \curly S^{n-1}\curly A	& \rTo	&\curly Q \curly S^{n-1}\curly A	&	&
\curly M	& 		& \rTo	& 		&\curly S^{n}\curly A		
\\
&\rdTo 	& 	&\ldTo 	&	&	&	&\rdTo	&\dTo	&\ldTo	&	&	&	&\rdTo	&	&	\ldTo &
\\												
&	& Y	&	&	&\lTo_{\pi_p}	&	&	&\text{Gr}^{r_p}(\curly S^{n-1}\curly A)	&	&	&\lTo_{\iota_p}	&
&	&Z	&	&
\end{diagram}
$$
where the notation $\text{Gr}^r(\curly F)$ denotes that the rank of the tautological quotient is $r$.

Starting with the class $1\in K(Z)$, Lemma \ref{lem:iota} implies that $(\iota_p)_{*}(1)$ is the product $(1-\curly M_\bullet/(\curly Q\curly S^{n-1}\curly A)_\bullet)$. Now for any family of variables $\mathbb T$, bundle $\curly F$, and Grassmannization $\Gr{s}{\curly F}$, one has the formal identity
\begin{equation}\label{eqn:rootsident}
{(1- \curly F_\bullet\mathbb T)}=(1-(\curly S\curly F)_\bullet\mathbb T){(1-(\curly Q\curly F)_\bullet\mathbb T)}
\end{equation}
and applying this many times, we can rewrite $(\iota_p)_{*}(1)$ as
\begin{equation}
\prod_{j=1}^\ell \frac{\displaystyle \left(1-\frac{(\curly B_j)_\bullet}{(\curly Q\curly S^{n-1}\curly A)_\bullet}\right)}{\displaystyle\prod_{k=1}^{n_j}\left(1-\frac{(\curly Q\curly S^{n_j-k}\curly B_j)_\bullet}{(\curly Q\curly S^{n-1}\curly A)_\bullet}\right)}.\nonumber
\end{equation}
Using Corollary \ref{cor:IRpfQonly} to compute $(\pi_p)_{*}$ of the above, we obtain that $(\rho_{i_p}^{r_p})_{*}(1)$ is given by
\begin{equation}
\IR_{\al{z}_p=0,\infty} \left( 
\prod_{j=1}^\ell \frac{(1-(\curly B_j)_\bullet \al{z}_p)}{(1-(\curly S^{n-1}\curly A)_\bullet \al{z}_p)} \frac{D_p}{\prod_{k=1}^{n_j}(1-(\curly Q\curly S^{n_j-k}\curly B_j)_\bullet \al{z}_p)}
\right),\nonumber
\end{equation}
but using Equation (\ref{eqn:rootsident}) on the denominator factors $(1-(\curly S^{n-1}\curly A)_\bullet \al{z}_p)$ this can also be rewritten as
\begin{equation}
\IR_{\al{z}_p=0,\infty} \left( 
R_p D_p \frac{\prod_{w=1}^n (1-(\curly Q\curly S^{n-w}\curly A)_\bullet \al{z}_p)}{\prod_{j=1}^\ell \prod_{k=1}^{n_j}(1-(\curly Q\curly S^{n_j-k}\curly B_j)_\bullet \al{z}_p)}
\right).
\end{equation}
Now observe that when the alphabets $\al{z}_{u}$ for $u<p$ are utilized as residue variables to push-forward classes containing only Grothendieck roots corresponding to tautological quotient bundles (as in Corollary \ref{cor:IRpfQonly}) through the rest of the Reineke resolution, the remaining rational function will produce exactly the interference factor $I_p$. The expression above depends only on bundles pulled back to $Y$ from earlier iterations of the Reineke construction, and so Lemma \ref{lem:iota} again applies. Furthermore, the formal algebraic manipulations required to compute each subsequent step in the resolution are completely analogous to those above, and therefore the result of the composition $(\rho_{\al{i}}^{\al{r}})_{*}=(\rho_{i_1}^{r_1})_{*}\compose \cdots \compose (\rho_{i_p}^{r_p})_{*}(1)$ is exactly the expression of Equation (\ref{eqn:main}). This proves Theorem \ref{thm:main}.

\section{Expansion in terms of Grothendieck polynomials}\label{s:GPIR}

Let $\lambda = (\lambda_1,\ldots,\lambda_r)$ be an integer sequence (not necessarily a partition) and $\curly A$ and $\curly B$ vector bundles of respective ranks $n$ and $p$. Let $\A = \{\alpha_i\}_{i=1}^{n}$ and $\mathbb B = \{\beta_j\}_{j=1}^p$ be sets of Grothendieck roots for $\curly A$ and $\curly B$ respectively. Let $\al{z} = \{z_1,\ldots,z_r\}$ and set $l=p-n$. Now define the factors
\begin{align*}
\mu_\lambda(\al{z})	&= \prod_{i=1}^r \left(1-{z_i}\right)^{\lambda_{i}-i} \\
\Delta(\al{z})	&= \prod_{1\leq i<j\leq r}\left(1- \frac{z_j}{z_i}\right) \\
P(\A,\mathbb B,\al{z})	&= \prod_{i=1}^r \frac{\prod_{b\in\mathbb B}\left( 1-b z_i \right)}{\left(1-{z_i}\right)^l \prod_{a\in\mathbb A}\left( 1- a z_i \right)}
\end{align*}
The \emph{double stable $g$-polynomial} $g_\lambda(\curly A - \curly B)$ corresponding to the integer sequence $\lambda$ is defined to be
\begin{equation}\label{eqn:Gdefn}
g_\lambda(\curly A-\curly B) = \IR_{\al{z}=0,\infty}\left( \mu_\lambda(\al{z}) \cdot \Delta(\al{z}) \cdot P(\A,\mathbb B,\al{z}) \cdot \dlog\al{z} \right).
\end{equation}

The definition above was pointed out to the author by Rim\'{a}nyi and Szenes, who promise a proof of the following conjecture (confirmed by the author in many computer experiments) in the upcoming paper \cite{kthom}.

\begin{conj}\label{conj:gisG}
For any vector bundles $\curly A$ and $\curly B$, and any integer sequence $\lambda=(\lambda_1,\ldots,\lambda_r)$, the double stable $g$-polynomial $g_\lambda(\curly A - \curly B)$ defined by Equation (\ref{eqn:Gdefn}) agrees with the double stable Grothendieck polynomial $G_\lambda(\curly A-\curly B)$ defined by Buch, e.g.\ in \cite[Equation~(7)]{BuchQCDT}.
\end{conj}

In all that follows, we assume the result of Conjecture \ref{conj:gisG} and use only the notation $G_\lambda$ for the (double) stable Grothendieck polynomials. Combining this with our main theorem, we obtain the following steps to expand the class $[\curly O_\Omega]$ in terms of the appropriate Grothendieck polynomials. Using the notation of Theorem \ref{thm:main},
\begin{itemize}
\item For each $i\in Q_0$ collect families of residue variables $\al{z}_k$ such that $i_k=i$, say $\al{z}_{j_1},\ldots,\al{z}_{j_\ell}$.
\item Combine these into the new families $\al{u}_i = \{u_{i1},u_{i2},\ldots,u_{in_i}\}= \al{z}_{j_1}\union\cdots\union\al{z}_{j_\ell}$ where $j_1<\cdots<j_\ell$ and observe that the numerators of the interference factors $I_k$ multiplied with the discriminant factors $D_k$ produce exactly the products $\D(\al{u}_i)$.
\item For each $i\in Q_0$ let $l_i=\rk(\curly E_i) - \rk(\curly M_i)$ and form the rational function $F(\al{u}_i)$ whose denominator is exactly the same as that of the product of all interference factors, but whose numerator is the product $$\prod_{i\in Q_0} \prod_{u\in \al{u}_{i}} \left( 1-u \right)^{-l_i}.$$
\item For all $i$ and $j$, substitute $u_{ij}=1-v_{ij}$ into $F$ and multiply by the factor $\prod_{i\in Q_0}\prod_{j=1}^{n_i} v_{ij}^j$ to form a new rational function $F'$.
\item Expand $F'$ as a Laurent series according the the convention that for any arrow $a\in Q_1$, $v_{t(a)j}<<v_{h(a)k}$ for any $j$ or $k$.
\item Finally, the expansion of $[\curly O_\Omega]$ in Grothendieck polynomials is obtained by interpreting the monomial $$\prod_{i\in Q_0} \al{v_i}^{\lambda_i} \leftrightsquigarrow \prod_{i\in Q_0} G_{\lambda_i}(\curly E_i - \curly M_i)$$ where for the integer sequence $\lambda_i = (\lambda_{i1},\ldots,\lambda_{in_i})$, $\al{v}_i^{\lambda_i}$ denotes the multi-index notation $\prod_{j=1}^{n_i} v_{ij}^{\lambda_{ij}}$, which we adopt throughout the sequel.
\end{itemize}

\begin{example}\label{ex:A2sp}Consider the $A_2$ quiver with vertices labeled $\{1\to 2\}$. Consider the orbit closure $\Omega_m(\curly E_\bullet)$ corresponding to $m_{11}=m_{12}=m_{22}=1$ and hence having dimension vector $(2,2)$. From the directed partition $\Phi^+ = \{\varphi_{22}\}\union\{\varphi_{12},\varphi_{11}\}$ one obtains the resolution pair $\al{i}=(2,1,2)$ and $\al{r}=(1,2,1)$. Following the recipe of Theorem \ref{thm:main}, set \begin{align*} \al{z}_1 &=\{x\} & \al{z}_1 &=\{y_1,y_2\} & \al{z}_3 &=\{z\}. \end{align*} Let $\curly E_\bullet\to X$ be a corresponding generic $Q$-bundle and set $\curly E_1=\curly A$, $\curly E_2=\curly B$, $\mathbb E_1=\{\alpha_1,\alpha_2\}$, $\mathbb E_2=\{\beta_1,\beta_2\}$. Notice this implies that $\mathbb M_1 = \{~\}$ and $\mathbb M_2 = \mathbb E_1 = \{\alpha_1,\alpha_2\}$. Applying the main theorem, we obtain that $[\curly O_{\Omega(\curly E_\bullet)}]$ is equal to applying the operation $\IR_{x=0,\infty}\IR_{y_2=0,\infty}\IR_{y_1=0,\infty}\IR_{z=0,\infty}$ to the differential form $$\left(\prod_{i=1}^2\frac{1-\alpha_i x}{1-\beta_i x}\right)\frac{\left(1-\frac{y_2}{y_1}\right)}{\prod_{i,j=1}^2 (1-\alpha_i y_j)}\left(\prod_{i=1}^2\frac{1-\alpha_i z}{1-\beta_i z}\right)\frac{\left(1-\frac{z}{x} \right)}{\prod_{j=1}^2\left(1-\frac{z}{y_j} \right)}\prod_{k=1}^3 \dlog\al{z}_k.$$ Renaming $x=u_1$ and $z=u_2$ and setting $\al{u}=\{u_1,u_2\}$ and $\al{y}=\al{z}_2=\{y_1,y_2\}$, this is further equal to $$P(\mathbb E_1,\mathbb M_1,\al{y})P(\mathbb E_2,\mathbb M_2,\al{u})\D(\al{y})\D(\al{u})\dlog\al{y}\dlog\al{u}$$ times the rational function $$\frac{1}{\prod_{i=1}^2(1-y_i)^2\prod_{j=1}^2\left(1-\frac{u_2}{y_j}\right)}.$$ Setting $a_i = 1-y_i$ and $b_i=1-u_i$ for $1\leq i \leq 2$, and multiplying the rational function above by $a_1a_2^2b_1b_2^2$ produces the rational function \begin{equation}\label{eqn:A2spQC}\frac{b_1(1-a_1)(1-a_2)}{a_1\left(1-\frac{a_1}{b_2}\right)\left(1-\frac{a_2}{b_2}\right)},\end{equation} and according to the itemized steps above, once this is expanded as a Laurent series, one can read off the quiver coefficients by interpreting $a^Ib^J \leftrightsquigarrow G_I(\curly A)G_J(\curly B-\curly A)$. Since $G_{I,J} = G_I$ whenever $J$ is a sequence of non-positive integers and $G_\emptyset = 1$ (see \cite[Section~3]{BuchGCQV}), the above rational function is equivalent (for our purposes) to the one obtained by setting $b_1=1$, namely the function $a_1^{-1}b_1$ and hence simply to $b_1$. This corresponds to the Grothendieck polynomial $G_1(\curly B-\curly A)$ and we conclude that the quiver efficient $c_{(\emptyset,(1))}(\Omega_m) = 1$ while all others are zero.\end{example}

\begin{example}
Consider the inbound $A_3$ quiver $\{1\to 2 \leftarrow 3\}$, and the same orbit and notation of Example \ref{ex:inA3MT}. Following the itemize list above, in Equation (\ref{eqn:inA3form}) set $t_1=x$, $t_2=z$, and $u_1=v$, $u_2=y_1$, and $u_3=y_2$ to obtain the families $\al{w}=\{w_1,w_2\}$, $\al{u}=\{u_1,u_2,u_3\}$, and $\al{t}=\{t_1,t_2\}$, associated to the vertices $1$, $2$, and $3$ respectively. In the new variables, one checks that $[\curly O_{\Omega(\curly E_\bullet)}]$ is given by applying the iterated residue operation $\IR_{\al{w}=0,\infty}\IR_{\al{t}=0,\infty}\IR_{\al{u}=0,\infty}$ to
$$P(\mathbb E_1,\mathbb M_1,\al{w})P(\mathbb E_2,\mathbb M_2,\al{u})P(\mathbb E_3,\mathbb M_3,\al{t})\Delta(\al{w})\Delta(\al{u})\Delta(\al{t})(\dlog\al{w})(\dlog\al{u})(\dlog\al{t})$$
times the rational function
$$\frac{\prod_{i=1}^3(1-u_i)}{\prod_{i=1}^2(1-w_i)^2\prod_{i=1}^2(1-t_i)^2}\prod_{i=2}^3\frac{1}{\prod_{s\in\{t_1\}\union\al{w}}\left(1-\frac{u_i}{s}\right)}.$$
The order of the residues above is important; in particular, the residues with respect to $\al{u}$ must be done first. In general, for each $a\in Q_1$ the residues with respect to variables corresponding to the vertex $t(a)$ must be computed before those corresponding to the vertex $h(a)$. Comparing the above with Equation (\ref{eqn:Gdefn}) and setting $a_i=1-w_i$, $b_j=1-u_j$, and $c_i=1-t_i$ for $1\leq i\leq 2$ and $1\leq j \leq 3$, observe that the quiver coefficients can be obtained by expanding the rational function 
$$\frac{\left(\prod_{i=1}^2 a_i^i\right)\left(\prod_{i=1}^3 b_i^i\right)\left(\prod_{i=1}^2 c_i^i\right)b_1b_2b_3(1-a_1)^2(1-a_2)^2(1-c_1)^2}{a_1^2a_2^2c_1^2c_2^2(b_2-a_1)(b_2-a_2)(b_2-c_1)(b_3-a_1)(b_3-a_2)(b_3-c_1)}$$
as a Laurent series and using the convention that $$a^Ib^Jc^K \leftrightsquigarrow G_I(\curly E_1)G_J(\curly E_2-\curly E_1\dirsum \curly E_3)G_K(\curly E_3).$$ We recommend rewriting the Laurent series above in the form
$$\frac{b_1^2b_3(1-a_1)^2(1-a_2)^2(1-c_1)^2}{a_1c_1\prod_{s\in\{b_2,b_3\}}\left(1-\frac{a_1}{s}\right)\left(1-\frac{a_2}{s}\right)\left(1-\frac{c_1}{s}\right)},$$
and expanding in the domain $a_j,c_1<<b_i$. In this example, the codimension of $\Omega_m$ is $3$ (cf. Equation (\ref{eqn:Gclass})) and we note that rational factor $b_1^2b_3/(a_1c_1)$ has odd degree. Thus, when the remaining factors are expanded, the signs alternate as desired. The difficulty is that most monomials in this expansion do not correspond to partitions and, as in the previous example, one must use a recursive recipe (see \cite[Equation~(3.1)]{BuchGCQV}) to expand these in the basis $\{G_\lambda\}$ for partitions $\lambda$, introducing new signs in a complicated way. Nonetheless a computation in \emph{Mathematica} confirms that the quiver coefficients are
\begin{align*}
c_{(\emptyset,(2,1),\emptyset)}(\Omega_m) &=1 & c_{((1),(2),\emptyset)}(\Omega_m) & = 1 & c_{((1),(2,1),\emptyset)}(\Omega_m) &=-1
\end{align*}
and all others are zero, which agrees with Equation (\ref{eqn:Gclass}).
\end{example}

\begin{example}[Giambelli-Thom-Porteous formula]
Consider again the $A_2$ quiver with vertices labeled $\{1\to 2\}$. Only now consider the general orbit closure $\Omega_m(\curly E_\bullet)$ corresponding to $m = (m_{11},m_{12},m_{22})$ and hence having dimension vector $(m_{11}+m_{12},m_{12}+m_{22})$. Let $\curly E_\bullet$ be a generic $Q$-bundle and write $e_1=\rk(\curly E_1)$ and $e_2=\rk(\curly E_2)$. From the directed partition $\Phi^+=\{\varphi_{22}\}\union\{\varphi_{12},\varphi_{11}\}$ one obtains the resolution pair $\al{i}=(2,1,2)$ and $\al{r}=(m_{22},e_1,m_{12})$. Observe that the composition of the first two mappings of the Reineke resolution $\rho_{1}^{e_1}\compose\rho_{2}^{m_{22}}$ is a homeomorphism since in the notation of Section \ref{s:PMT}, it represents the sequence of maps
$$
\begin{diagram}[2em]
Z(\curly S \curly E_1\to \curly Q\curly S\curly E_2) & \rTo & \Gr{0}{\curly S\curly E_2} &\rTo &
Z(0\to \curly Q\curly E_1) & \rTo & \Gr{0}{\curly E_1} & \rTo Z(\curly E_1\to\curly Q\curly E_2)
\end{diagram}
$$
and $\curly S \curly E_1$ has rank zero. Hence we need only to compute the image $(\rho_{2}^{m_{22}})_{*}(1)$ and this is equivalent to applying Theorem \ref{thm:main} to the updated resolution pair $\al{i}=(2), \al{r}=(m_{22})$.
The fact that this computation simplifies is related to the fact that in Example \ref{ex:A2sp}, the rational function (\ref{eqn:A2spQC}) can be simplified to a monomial by setting $b_2=1$. We obtain that $$[\curly O_{\Omega(\curly E_\bullet)}] = \IR_{\al{z}=0,\infty}\left(\frac{\left(1 - (\curly E_1)_\bullet\al{z}_\bullet)\right)}{\left(1-(\curly E_2)_\bullet\al{z}_\bullet \right)}\D(\al{z})\dlog\al{z}\right)$$ where $\al{z}=(z_1,\ldots,z_{m_{22}})$.

Following the itemized steps above, we set $l = e_2 - e_1$ and consider the product $\prod_{i=1}^{m_{22}}(1-u_i)^{-l}$ and finally the monomial $\prod_{i=1}^{m_{22}} v_i^{i-l}$ to obtain that $$[\curly O_{\Omega(\curly E_\bullet)}] = G_{(1-l,2-l,\ldots,m_{22}-l)}(\curly E_2 - \curly E_1).$$ Notice that the integer sequence above is strictly increasing and therefore not a partition. However, $G_{I,p-1,p,J}=G_{I,p,p,J}$ for any integer sequences $I$ and $J$ and any integer $p$ (see \cite[Section~3]{BuchGCQV}), and so applying this iteratively above yields the Grothendieck polynomial $G_{R}(\curly E_2-\curly E_1)$ where $R$ is the rectangular partition $(m_{22}-l)^{m_{22}}$. In Example \ref{ex:A2sp} this corresponded to the step $a_1^{-1}b_1 \leadsto b_1$. Finally, if one sets $r=m_{12}$, this has the pleasing form $(e_1-r)^{(e_2-r)}$ (cf.\ \cite[Theorem~2.3]{BuchGCQV}). One thinks of ``$r$" denoting the rank of the map $f:\curly E_1\to \curly E_2$ since after all $\Omega(\curly E_\bullet)$ is actually the degeneracy locus $\{x \in X \st \rk(f) \leq m_{12}\}$. We conclude that the quiver coefficient $c_{(\emptyset,R)}(\Omega) = 1$ and all others are zero.
\end{example}

\bibliographystyle{amsalpha}
\bibliography{KIRbib}

\providecommand{\bysame}{\leavevmode\hbox to3em{\hrulefill}\thinspace}
\providecommand{\MR}{\relax\ifhmode\unskip\space\fi MR }
\providecommand{\MRhref}[2]{%
  \href{http://www.ams.org/mathscinet-getitem?mr=#1}{#2}
}
\providecommand{\href}[2]{#2}
\begin{thebibliography}{Rim13b}

\bibitem[BS12]{BS}
G.~B\'{e}rczi and A.~Szenes, \emph{{T}hom polynomials of {M}orin
  singularities}, Ann. of {M}ath. \textbf{175} (2012), no.~2, 567--629.
  \MR{2877067}

\bibitem[Buc02a]{BuchGCQV}
A.S. Buch, \emph{{G}rothendieck classes of quiver varieties}, Duke Math. J.
  \textbf{115} (2002), no.~1, 75--103. \MR{1932326}

\bibitem[Buc02b]{BuchKLR}
\bysame, \emph{A {L}ittlewood-{R}ichardson rule for the {$K$}-theory of
  {G}rassmannians}, Acta Math. \textbf{189} (2002), no.~1, 37--78. \MR{1946917}

\bibitem[Buc05]{BuchCKT}
\bysame, \emph{Combinatorial {$K$}-theory}, Topics in cohomological studies of
  algebraic varieties, Trends Math., Birkh\"{a}user, Basel, 2005, pp.~87--103.
  \MR{2143073}

\bibitem[Buc08]{BuchQCDT}
\bysame, \emph{Quiver coefficients of {D}ynkin type}, Michigan Math. J.
  \textbf{57} (2008), 93--120, Special volume in honor of {M}elvin {H}ochster.
  \MR{2492443}

\bibitem[BZ01]{BZA}
G.~Bobi\'{n}ski and G.~Zwara, \emph{Normality of orbit closures for {D}ynkin
  quivers of type {$\mathbb {A}_n$} ({E}nglish summary)}, Manuscripta Math.
  \textbf{105} (2001), no.~1, 103--109. \MR{2215120}

\bibitem[BZ02]{BZD}
\bysame, \emph{{S}chubert varieties and representations of {D}ynkin quivers
  ({E}nglish summary)}, Colloq. Math. \textbf{94} (2002), no.~2, 285--309.
  \MR{1967381}

\bibitem[CG97]{CG}
N.~Chriss and V.~Ginzburg, \emph{Representation theory and complex geometry},
  Birkh\"{a}user Boston, Inc., Boston, MA, 1997. \MR{1433132}

\bibitem[FR07]{FRLMS}
L.M. Feh\'{e}r and R.~Rim\'{a}nyi, \emph{On the structure of {T}hom polynomials
  of singularities}, Bull. Lond. Math. Soc. \textbf{39} (2007), no.~4,
  541--549. \MR{2346933}

\bibitem[FR12]{FRAnn}
\bysame, \emph{Thom series of contact singularities}, Ann. of Math. \textbf{(2)
  176} (2012), no.~3, 1381--1426. \MR{2979854}

\bibitem[FS12]{FSM}
A.~Fink and D.E. Speyer, \emph{${K}$-classes for matroids and equivariant
  localization}, Duke Math. J. \textbf{161} (2012), no.~14, 2699--2723.
  \MR{2993138}

\bibitem[Gab72]{Gab}
P.~Gabriel, \emph{Unzerlegbare {D}arstellungen. {I} ({G}erman. {E}nglish
  summary)}, Manuscripta {M}ath. \textbf{6} (1972), 71--103, correction, ibid.
  6 (1972), 309. \MR{332887}

\bibitem[Kal13]{Kal}
R.~Kaliszewski, \emph{Structure of quiver polynomials and {S}chur positivity},
  Ph.D. thesis, UNC--Chapel Hill, 2013.

\bibitem[Kaz10a]{KazGHD}
M.~Kazarian, \emph{{G}ysin homomorphism and degeneracies}, unpublished, 2010.

\bibitem[Kaz10b]{KazTP}
\bysame, \emph{Non-associative {H}ilbert scheme and {T}hom polynomials},
  unpublished, 2010.

\bibitem[KR99]{EqKT}
A.~Knutson and I.~Rosu, \emph{Equivariant ${K}$-theory and equivariant
  cohomology}, Math. Z. (1999), no.~243, 423--448, appendix. \MR{1970011}

\bibitem[KS11]{KScoha}
M.~Kontsevich and Y.~Soibelman, \emph{Cohomological {H}all algebra, exponential
  {H}odge structures and motivic {D}onaldson-{T}homas invariants}, Commun.
  Number Theory Phys. \textbf{5} (2011), no.~2, 231--352. \MR{2851153}

\bibitem[LS82]{LaSch}
A.~Lascoux and M.P. Sch\"{u}tzenberger, \emph{Structure de {H}opf de l'anneau
  de cohomologie et de l'anneau de {G}rothendieck d'une variŽtŽ de drapeaux
  ({F}rench. {E}nglish summary)}, C. R. Acad. Sci. Paris S\'{e}r. I Math.
  \textbf{295} (1982), no.~11, 629--633. \MR{0686357}

\bibitem[Rei03]{Re}
M.~Reineke, \emph{Quivers, desingularizations and canonical bases}, Studies in
  memory of Issai Schur (Chevaleret/Rehovot, 2000), Progr. Math., vol. 210,
  Birkh\"{a}user, Boston, MA, 2003, pp.~325--344. \MR{1985731}

\bibitem[Rim13a]{coha}
R.~Rim\'{a}nyi, \emph{On the cohomological {H}all algebra of {D}ynkin quivers},
  preprint, 2013.

\bibitem[Rim13b]{qr3}
\bysame, \emph{Quiver polynomials in iterated residue form}, preprint, 2013.

\bibitem[RS13]{kthom}
R.~Rim\'{a}nyi and A.~Szenes, \emph{${K}$-theoretic {T}hom polynomials and
  their expansions in {G}rothendieck polynomials}, in preparation, 2013.

\bibitem[Zie12]{Z}
M.~Zielenkiewicz, \emph{Integration over homogenous spaces for classical {L}ie
  groups using iterated residues at infinity}, preprint, 2012.

\end{thebibliography}

\end{document}